\documentclass[12pt]{article}
\usepackage{amsmath}
\usepackage{amssymb}
\usepackage{amsmath,amssymb,textcomp,graphicx,amsthm} 
\usepackage{amssymb,verbatim,color}
\newtheorem{thm}{Theorem}
\newtheorem{lemma}[thm]{Lemma}
\newtheorem{cor}[thm]{Corollary}
\newtheorem{prop}[thm]{Proposition}

\newcommand{\V}{V_{\infty}}
\DeclareMathOperator{\dist}{dist}
\begin{document}
\title{Infinity-Harmonic Potentials and Their Streamlines}
\author{Erik Lindgren, Peter Lindqvist}

\date{\today}
\maketitle

\medskip
    {\small \textsc{Abstract:} \textsf{
We consider certain solutions of the Infinity-Laplace Equation in planar convex rings. Their ascending streamlines are unique while the descending ones may bifurcate. We prove that bifurcation occurs in the generic situation and as a consequence, the solutions cannot have Lipschitz continuous gradients.}
      
\bigskip 

{\small \textsf{AMS Classification 2010}: 49N60, 35J15, 35J60, 35J65, 35J70.} 

\section{Introduction}

The solutions of the celebrated $\infty$-Laplace Equation
$$
 \Delta_{\infty}u\,\equiv\,\sum_{i,j}\frac{\partial u}{\partial x_i} \frac{\partial u}{\partial x_j}\frac{\partial^2 u}{\partial x_i \partial x_j}\,=\,0,
 $$
 which is the formal limit of the $p$-Laplace Equations
$$
    \Delta_p u\,\equiv\,  \nabla\! \cdot \! (|\nabla
    u|^{p-2}\nabla u)\,=\,0
  $$
  as $p \to \infty$,  have many fascinating properties. The solutions provide the best Lipschitz extension of their boundary values (see \cite{A1}) and the equation appears even in Stochastic Game Theory (see \cite{PSW}).

  A characteristic feature for classical solutions is that \emph{the speed} $|\nabla u|$ \emph{is constant along a streamline}, which is a useful property for applications to image processing, see \cite{CMS}. Indeed, along the streamline $x = x(t)$ with the equation
  $$
  \frac{d x}{d t}\,=\, \nabla u(x(t))
  $$
  we should have
  $$\frac{d \,}{d t}|\nabla u(x(t))|^2\,
  %=\, 2 \sum\frac{\partial u}{\partial x_i} \frac{\partial^2 u}{\partial x_i \partial x_j}\frac{d x_j}{d t}\,
  =\, 2 \Delta_{\infty}u(x(t))\,=\,0
  $$
  so that
  $$|\nabla u(x(t))| \,=\, \text{constant}.$$
  However, the calculation requires second partial derivatives. We shall see that this interpretation of constant speed often fails.

  The solutions of the $\infty$-Laplace Equation, the so-called $\infty$-\emph{harmonic functions}, are defined in the viscosity sense as in \cite{J}, \cite{JLM} and \cite{S}. They are continuous and even differentiable. O. Savin \cite{S} has proved that in the plane their gradient is continuous and even locally H\"{o}lder continuous, according to \cite{ESa}. Thus the solutions are of class $C^{1,\alpha}_{loc}$ in the two dimensional case. In \cite{KZZ} the speed $|\nabla u|$ is shown to belong to a Sobolev space. In higher dimensions the gradient exists (in the classical sense) at every point by a result of L. Evans and Ch. Smart, cf. \cite{ES}. At the moment of writing, the $C^1_{loc}$-property is not known in higher dimensions. This unsettled urgent question is the reason for why we restrict our exposition to two dimensions. In the plane the equation reads
  $$
  \Bigl(\frac{\partial u}{\partial x_1}\Bigr)^{\!2}\frac{\partial ^2u}{\partial x_1^2}\,+\,2\, \frac{\partial u}{\partial x_1}\frac{\partial u}{\partial x_2}\frac{\partial ^2u}{\partial x_1 \partial x_2}\,+\, \Bigl(\frac{\partial u}{\partial x_2}\Bigr)^{\!2}\frac{\partial ^2u}{\partial x_2^2}\,=\,0
  $$
  as in G. Aronsson's work \cite{A2} about the streamlines.

 \paragraph{Notation.} We fix some notation. Suppose that $\Omega$ is a \emph{convex} bounded domain in the plane $\mathbb{R}^2$ containing a compact \emph{convex} set $K$ with boundary $\Gamma = \partial K.$ The case when $K$ reduces to a single point is of special interest. The domain $G = \Omega \setminus K$ is a ``convex ring''; it has the outer boundary $\partial \Omega$ and the inner boundary $\Gamma$. The object of our work is the Dirichlet boundary value problem
  \begin{equation}\label{maineq}
    \begin{cases} \Delta_{\infty}u\,=\,0\qquad\text{in}\qquad G\\ \phantom{ \Delta_{\infty}}
      u\,=\,0\qquad\text{on}\qquad \partial \Omega\\\phantom {\Delta_{\infty}}
      u\,=\,1 \qquad\text{on}\qquad \Gamma.
    \end{cases}
  \end{equation}
  The unique solution, say $V_{\infty}$, attains the boundary values in the classical sense (this holds for all domains, whether they are convex or not). Hence
  $$V_{\infty} \in C(\overline{G})\quad \text{where} \quad \overline{G}\, =\, \partial \Omega \cup G\cup \Gamma.$$

  \paragraph{Some properties.}
By the Maximum Principle, $0 < \V < 1$ in $G$. (It is convenient to put  $\V =1$ in $K$ and $=0$ outside $\Omega$.)  The gradient  $\nabla V_{\infty} \in C^{\alpha}_{loc}(G)$ for some small $\alpha$, cf. \cite{ESa}. We use some fundamental properties valid in convex rings, which are due to J. Lewis \cite{L}. See also \cite{Ja}. We need the following
  \begin{itemize}
  \item The level sets $\{V_{\infty}(x) > c\}$ are convex, $0 \leq c < 1.$
  \item $\Delta_{p}V_{\infty}\,\equiv\, \nabla\!\cdot\!\bigl(|\nabla V_{\infty}|^{p-2} \nabla V_{\infty}\bigr)\,\leq 0$ when $p \geq 2.$
  \item $\nabla V_{\infty} \neq 0$ in $G$.
  \end{itemize}
  We interpret the inequality $\Delta_{p}V_{\infty}\,\leq\,0$ in the viscosity sense. This is equivalent to the usual definition of $p$-superharmonic functions, cf. \cite{JLM}, \cite{JJ}. In particular ``$\Delta V_{\infty} \,\leq\,0$'' and so $ V_{\infty} $ \emph{is an ordinary superharmonic function.}

 \paragraph{Streamlines.} Let us return to the \emph{ascending} streamlines $x =x(t).$ They are the trajectories of the gradient flow
  \begin{equation}\label{asc}
    \begin{cases}
    \dfrac{dx}{dt}\,=\,\nabla V_{\infty}(x(t)),\quad t > t_0,\\
    x(t_0)\, =\, x_0 \in G\cup \partial \Omega
    \end{cases}
\end{equation}
and intersect the convex level curves orthogonally. (If the initial point $x_0 \in \partial \Omega$ and $\nabla  V_{\infty}(x(t_0)) =0,$ some special care is needed.) By Peano's Existence Theorem, \emph{there exists at least one solution starting at} $x_0.$ Since $\nabla V_{\infty}\neq0$, the trajectory cannot terminate inside $G$. In fact, $x(t) \in G$ when $t_0\leq t<T$ for some finite $T$ and $x(T) \in \Gamma.$ One of our main results is that the solution is unique.

\begin{thm}[Ascending uniqueness]\label{uniqueness} The solution to the equation   (\ref{asc}) of the ascending gradient flow  is unique and terminates at $\Gamma$.
\end{thm}

Despite uniqueness, two trajectories, starting at different points, can meet and join. But the trajectories cannot cross. \emph{The first point at which two streamlines meet} (after which they become a joint trajectory) {is here called a Cl-point}. Notice that uniqueness is not valid for the usual \emph{descending} streamlines coming from the equation 
$$
\frac{dx}{dt}\,=\,-\,\nabla V_{\infty}(x(t))
$$
with a \emph{minus} sign! They allow bifurcation. The proof of the uniqueness theorem is delicate, since the Picard-Lindel\"{o}f Theorem is not applicable, when $\nabla V_{\infty}$ is not Lipschitz continuous. (Mere H\"{o}lder continuity is not sufficient.) We base our reasoning on the expedient inequality
\begin{equation}\label{expedient}
  \oint_{\partial D}|\nabla V_{\infty}|^{p-2}\langle\nabla V_{\infty},\mathbf{n}\rangle\,ds\,\leq\, 0,\qquad p \geq 2,
\end{equation} valid for any domain $D \subset \subset G$ with Lipschitz boundary $\partial D$. Here $\mathbf{n}$ denotes the outer unit normal. The proof given in Proposition \ref{fundamental} requires several regularizations so that the inequality $\Delta_p V_{\infty}\,\leq\,0$ can be used pointwise as in \cite{JJ}. The difficulty is the absence of second derivatives.

Our next theorem provides a tricky device for detecting Cl-points.

\begin{thm}\label{Cl} Let $\xi_0 \in \partial \Omega$ and denote
  $$\alpha\,=\,\limsup_{x \to \xi_0}|\nabla  V_{\infty}(x)|.$$
  Assume that
  $$ \beta \,\leq \,\liminf_{x \to \xi}|\nabla  V_{\infty}(x)|\quad\text{whenever}\quad \xi \in \Gamma.$$
  If $\beta > \alpha$, then there exists a neighborhood of $\xi_0$ such that every pair of streamlines starting there will meet before reaching $\Gamma$.
\end{thm}

In general, we have not succeeded in proving that the speed $|\nabla V_{\infty}(x(t))|$ is non-decreasing along the streamline. Thus the use of the theorem is somewhat elaborate. Let us mention some immediate consequences. First, the fact that two streamlines meet means that the \emph{descending} gradient flow does not have unique solutions. By the Picard-Lindel\"{o}f Theorem the function $\mathbf{-}V_{\infty}$ cannot therefore belong to the class $C^{1,1}_{loc}(G)$ in the presence of Cl-points. By general theory, the \emph{descending} gradient flow $\frac{dx}{dt} = -\nabla u(x)$ has a unique solution if $u$ is locally semiconvex . It follows that our $V_{\infty}$ cannot be locally semi\emph{convex}\footnote{A function $f$ is semiconvex if $f(x)+C|x|^2$ is convex for some constant $C>0$.}. (Neither can $\psi(V_{\infty})$ be for a smooth strictly monotone function $\psi$, since $V_{\infty}$ and $\psi(V_{\infty})$ have the same level sets.)

To apply the theorem we notice that it is always possible to choose $\beta > 0$,  see Lemma \ref{topp}. Thus, if we can find a point $\xi_0 \in \partial \Omega$ yielding $\alpha = 0$, we have obtained the inequality $\beta > \alpha$. According to a result in \cite{MPS} the following holds in convex domains in the plane: if the boundary has an irregular boundary point which is  a corner with interior angle less than $\pi$, then $|\nabla V_{\infty}| = 0$ at the corner. This provides an $\alpha = 0$.

\begin{thm} If $\partial \Omega$ has  a corner with angle less than $\pi$, then there are streamlines that meet in $G$ before reaching $\Gamma$. In particular, $V_{\infty}$ is not of class $C^{1,1}_{loc}(G)$.
\end{thm}

For  a special kind of domains  the distance function $\mathrm{dist}(x,\partial \Omega)$  is the $\infty$-potential. A \emph{stadium} is a domain where the distance function attains its maximum value at all its  singular points. These sets have a simple characterization in the plane. Namely, 
$$
H = \{x|\,\dist(x,\partial \Omega)=\|\dist(x,\partial \Omega)\|_{\infty}\},
$$
$$
\Omega=\{x|\, \dist(x,H)<\|\dist(x,\partial \Omega)\|_{\infty}\}.$$
See Theorem 6 in \cite{CF}. The set $H$ is called the \emph{High Ridge}. The simplest example of a stadium is the unit disk:
$$
H = \{0\}, \quad \Omega = \{x|\,0 < |x| < 1\}.
$$
In a stadium, when $\Gamma$ is the High Ridge, the solution is smooth and no streamlines meet. We argue that all other convex rings have Cl-points. If $\Gamma$ is a single point we have the following theorem.

\begin{thm}\label{single} Assume that $\Gamma$ is a single point. If $\Omega$ is not a disk centered at $\Gamma$, then there are  streamlines that meet. In fact, all streamlines that are not entirely inside the closed disk with radius $\mathrm{dist}(\Gamma, \partial \Omega)$ centered at $\Gamma$ have Cl-points.   In particular, $V_{\infty}$ is not $C^{1,1}_{loc}(G)$.
\end{thm}
That $V_{\infty}$ is not of class $C^{2}_{loc}(G)$ has been proved before, see Corollary 1.2 in \cite{SWY}. See also Corollary 23 in \cite{CF2} for a related result. Also the case when $\Gamma$ is a subset of the \emph{High Ridge} (though the domain is not necessarily a stadium) is accessible.
\begin{thm}\label{hrthm}
Suppose 	$\Gamma$ is a subset of the High Ridge of $\Omega$. Unless $\Omega$ is a stadium and $\Gamma$ its High Ridge, there are streamlines that meet.  In particular,  $V_{\infty}$ is not $C^{1,1}_{loc}(G)$.
\end{thm}

We also mention that Theorem \ref{single} reveals a queer \emph{instability} for the $\infty$-Laplace equation. Indeed, the solution of \eqref{maineq} in the disk $0 < |x| < 1$ is smooth, while the  corresponding solution
%to \eqref{maineq}
in an ellipse exhibits points where the second order derivatives are not bounded. After a coordinate  transformation, this implies that in a disk the solution  of \eqref{maineq}  with $\Delta_\infty $  replaced by the operator 
$$
u_x^2u_{xx}+2(1+\delta)u_xu_yu_{xy}+(1+\delta)^2u_y^2 u_{yy}
$$
exhibits this kind of singularites for any $\delta>0$, but not for $\delta=0$. A similar instability occurs if the midpoint of the disk is perturbed.

We conclude our work with some remarks about  a square. This is a challenging example, indeed. Now the domain $\Omega$ is a square and $\Gamma$ is its midpoint. In this case the gradient $\nabla\V$ is continuous also on the sides, but $\nabla \V = 0$ at the four corners (and only there), which gives an $\alpha = 0$ for free in Theorem \ref{Cl}. By symmetry the diagonals are streamlines, so are the medians. It seems as if all the streamlines, except the  four medians, would join a diagonal before reaching the midpoint (see Figure 1). We record three results.

First, we show that there are infinitely many Cl-points near the corners. Second, we show that also near the origin there are are infinitely many Cl-points. Finally, we argue that all the streamlines, except  the medians, do have infinitely many Cl-points. (It seems as if all points on the diagonals were Cl-points and that these are the only Cl-points.) It is likely that the $\infty$-harmonic potential function is related to the $\infty$-eigenvalue problem, introduced  in \cite{JLM3}. Indeed, this resemblance was the starting point of our investigation.

The reader is supposed to be familiar with the $\infty$-Laplacian. For the concept of viscosity solutions we refer to \cite{K} and \cite{CIL}. We use standard notation. We restrict ourselves to the plane, but most of our exposition is valid even in higher dimensions provided that the gradient $\nabla \V$ be continuous.

\section{Preliminaries}

A fundamental tool is inequality (\ref{expedient}) for line integrals. For smooth functions it comes from an integration by parts. We shall use the method in \cite{JJ}.

\begin{prop} \label{fundamental} Let $p \geq 2$ and assume that $D \subset \subset G$ has a Lipschitz boundary $\partial D$. Then
  \begin{equation}\label{normal}
   \boxed{ \oint_{\partial D} |\nabla \V|^{p-2}\langle \nabla  \V, \mathbf{n}\rangle\,ds\,\leq\,0}
  \end{equation}
  where $\mathbf{n}$ is the outer unit normal.
\end{prop}

\bigskip
\begin{proof} Due to the lack of second derivatives we use two regularizations.

\bigskip
\emph{Step 1}. Let $V_{\infty,\varepsilon}$ be the infimal convolution
$$ V_{\infty,\varepsilon}(x)\,=\,\inf_{y \in G}\Bigl\{\V(y) +\frac{|x-y|^2}{2\varepsilon}\Bigr\}.$$
By standard theory $V_{\infty,\varepsilon} \nearrow \V$ locally uniformly in $G$ and
\begin{equation}\label{Ale}
  \Delta_{p}V_{\infty,\varepsilon}\,\leq\,0\qquad \text{in}\qquad D
\end{equation}
in the viscosity sense, when $\varepsilon >0$ is small enough. The fact that $\Delta_{p}\V\,\leq\,0 $ implies this. Furthermore, the function
$$
V_{\infty,\varepsilon}(x)\,-\,\frac{|x|^2}{2\varepsilon}$$
is concave. Therefore it has second derivatives in the sense of Alexandroff a.e. So does
$V_{\infty,\varepsilon}.$ It follows that inequality (\ref{Ale}) holds almost everywhere, when the second derivatives are taken in Alexandroff's sense. At almost every $x \in D$
\begin{align*}
  V_{\infty,\varepsilon}(y)\,=\, V_{\infty,\varepsilon}(x) + \langle\nabla V_{\infty,\varepsilon}(x),y-x\rangle\\
  +\frac{1}{2}\langle y-x,\mathbb{D}^2V_{\infty,\varepsilon}(x)(y-x)\rangle + o(|x-y|^2)
\end{align*}
as $y \to x.$ Here $\mathbb{D}^2V_{\infty,\varepsilon}$ is the Hessian matrix of second Alexandroff derivatives.

\bigskip
\emph{Step 2}. We claim that
$$\nabla V_{\infty,\varepsilon} \rightarrow \nabla \V\qquad$$
a.e. in $D$, as $\varepsilon \to 0$. Since $V_{\infty,\varepsilon}$ is Lipschitz continuous, it is differentiable almost everywhere. Fix a point $x\in D$ at which $\nabla V_{\infty,\varepsilon}(x)$ exists. The infimum is attained at a point $x_{\varepsilon}$ in $G$:
$$  V_{\infty,\varepsilon}(x)\,=\,\V(x_{\varepsilon}) +\frac{|x-x_{\varepsilon}|^2}{2\varepsilon}.$$ It is easy to see that
\begin{equation}\label{easy}
\nabla V_{\infty ,\varepsilon}(x)\,=\,\nabla\V(x_{\varepsilon}).
\end{equation}
Indeed, 
$$
V_{\infty,\varepsilon}(x+h)-V_{\infty,\varepsilon}(x)\leq V_\infty(y)+\frac{|x+h-y|^2}{2\varepsilon}-V_\infty(x_\varepsilon), 
$$
provided that $x+h$ and $y$ are in $G$. The choice $y=x_\varepsilon+h$ yields
$$
V_{\infty,\varepsilon}(x+h)-V_{\infty,\varepsilon}(x)\leq V_\infty(x_\varepsilon+h)-V_\infty(x_\varepsilon).
$$
Write $h=t\bf e$, $t>0$, where $\bf e$ is a unit vector.  Divide by $t$ and let $t\to 0^ +$ to see that
$$
\langle \nabla V_{\infty,\varepsilon}(x), {\bf e}\rangle \leq  \langle \nabla V_{\infty}(x_\varepsilon), \bf e\rangle.
$$
Since $\bf e$ was arbitary, \eqref{easy} follows. The convergence at $x$ now follows from 
\begin{align}
\nonumber|\nabla V_{\infty,\varepsilon}(x) - \nabla \V(x)|\,&=\,|\nabla \V(x_{\varepsilon})-\nabla\V(x)|\,\\
&\label{epsconv}\leq\, C_D|x-x_{\varepsilon}|^{\alpha}\\
&\leq C_D\varepsilon^{\alpha/2} \rightarrow 0,\nonumber 
\end{align}
as $\varepsilon \to 0$, upon renaming the constant, since $\nabla \V$ is locally  H\"{o}lder continuous in $G$. Thus \eqref{epsconv} holds at a.e. point $x$.

We also note that 
$$
\nabla V_{\infty,\varepsilon}(x)=\frac{x-x_\varepsilon}{\varepsilon}=\nabla \V(x_\varepsilon)
$$
necessarily holds at a point of differentiability. Therefore, $x_\varepsilon$ is unique at such a point.

From \eqref{easy} we also get the uniform bound 
$$
\|\nabla V_{\infty,\varepsilon}\|_{L^\infty(D)}\leq \|\nabla V_{\infty}\|_{L^\infty(G)}, 
$$
which will be needed.
% Also $|x-x_{\varepsilon}| \leq \sqrt{2\varepsilon}.$ In other words, $\nabla V_{\infty,\varepsilon} \to \nabla \V$ uniformly in $\overline{D}.$

\bigskip
\emph{Step 3}. To obtain second derivatives we define the convolution
$$V_{\infty,\varepsilon,j}\,=\,V_{\infty,\varepsilon}\star\rho_j$$
where $\rho_j$ is a standard mollifier. Since \eqref{epsconv} holds a.e., the following estimate follows from a standard argument
\begin{equation}
\label{peterest}
\|\nabla V_{\infty,\varepsilon,j}(x)-\nabla V_{\infty,j}(x)\|_{L^\infty(D)}\leq C\varepsilon^{\alpha/2}, 
\end{equation}
for some $\alpha>0$.

By the proof of Alexandroff's Theorem in \cite{EG}
$$ \mathbb{D}^2V_{\infty,\varepsilon}\,=\,\lim_{j \to \infty}\bigl(D^2(V_{\infty,\varepsilon}\star \rho_j)\bigr)$$
almost everywhere. Thus
$$\lim_{j \to \infty}\Delta_pV_{\infty,\varepsilon,j}\,=\,\Delta_pV_{\infty,\varepsilon}$$
almost everywhere in $D$; the second derivatives are in the sense of Alexandroff. The convolution preserves concavity: 
$$
D^2V_{\infty,\varepsilon,j}\,\leq\,\frac{I_2}{\varepsilon},\qquad \Delta V_{\infty,\varepsilon,j}\,\leq \,\frac{2}{\varepsilon}$$
where $I_2$ is the identity matrix. It is immediate that
$$|\nabla V_{\infty,\varepsilon,j}|\,\leq\, \|\nabla V_{\infty,\varepsilon} \|_{\infty,D}\,\leq\,\|\nabla \V\|_{\infty,G}\,=\,C.$$
Together, these inequalities yield the bound
$$
  -\Delta_p  V_{\infty,\varepsilon,j}\,\geq\,-C^{p-2}\,\frac{2+(p-2)}{\varepsilon}.
$$
Thus we can use Fatou's Lemma to obtain
\begin{align}\nonumber 
  \liminf_{j \to \infty}&\displaystyle\iint_D( -\Delta_p  V_{\infty,\varepsilon,j})\,dx_1dx_2\\\label{limsupineq}
  \geq&\iint_D \liminf_{j \to \infty} ( -\Delta_p  V_{\infty,\varepsilon,j})\,dx_1dx_2\\
  =&\iint_D( -\Delta_p  V_{\infty,\varepsilon})\,dx_1dx_2
  \geq \iint_D 0\,dx_1dx_2\,=\,0,\nonumber 
  \end{align}
where inequality (\ref{Ale}) was used at the end.

\bigskip
\emph{Step 4}. By the Divergence Theorem
$$\oint_{\partial D}|\nabla V_{\infty,\varepsilon,j}|^{p-2}\langle \nabla V_{\infty,\varepsilon,j},\mathbf{n}\rangle\,ds\,=\,\iint_D\Delta_p  V_{\infty,\varepsilon,j}\,dx_1dx_2.$$
By \eqref{peterest}, 
$$
\oint_{\partial D}|\nabla V_{\infty,\varepsilon,j}|^{p-2}\langle \nabla V_{\infty,\varepsilon,j},\mathbf{n}\rangle\,ds=\oint_{\partial D}|\nabla V_{\infty,j}|^{p-2}\langle \nabla V_{\infty,j},\mathbf{n}\rangle\,ds+\mathcal{O}(\varepsilon^{\alpha/2}).
$$
Therefore, since $\nabla V_{\infty,j}\to \nabla V_\infty$ uniformly, 
\begin{align*}
\displaystyle \oint_{\partial D}|\nabla V_{\infty}|^{p-2}\langle \nabla V_{\infty},\mathbf{n}\rangle\,ds+\mathcal{O}(\varepsilon^{\alpha/2})&\leq \displaystyle \limsup_{j\to\infty}\oint_{\partial D}|\nabla V_{\infty,\varepsilon,j}|^{p-2}\langle \nabla V_{\infty,\varepsilon,j},\mathbf{n}\rangle\,ds\\&\leq 0, 
\end{align*}
by \eqref{limsupineq}. Since $\varepsilon$ is arbitrary, the proposition follows.

\end{proof}

\bigskip
The function
$$
W_{\infty}\,=\log(\V)
$$
is often more convenient. It has the same level curves and streamlines as $\V$. Under the same assumptions as in Proposition $\ref{fundamental}$ we have
$$
 \boxed{ -(p-1)\iint_D |\nabla W_{\infty}|^{p}\,dx_1dx_2\,\geq \,
  \oint_{\partial D}|\nabla W_{\infty}|^{p-2}\langle\nabla W_{\infty},\mathbf{n}\rangle ds.}
$$
The proof is similar, since
$$\Delta_pv + (p-1)|\nabla v|^p \,=\,\frac{\Delta_pu}{u^{p+1}},\qquad v = \log(u)$$
holds for smooth functions $u > 0$.
\bigskip

\section{Estimates for the Gradient}

\begin{lemma}\label{nonzero} We have
  $$0 < |\nabla \V(x)| \leq \frac{1}{\mathrm{dist}(\Gamma,\partial \Omega)}\quad \text{when}\quad x \in G.$$
\end{lemma}

\bigskip

\begin{proof} That $\nabla \V \neq 0$ is proved in \cite{L}, see also \cite{Ja}. This is a simple consequence of the convexity  of the level curves.\footnote{Actually, one has
    $$|\nabla\V(x)|\,\geq\,|\V(x)|\mathrm{diam}(\Omega)^{-1}.$$}

  Since $\V$ is an optimal extension of its boundary values,
  $$\|\nabla \V\|_{\infty,G}\,\leq\,\|\nabla v\|_{\infty,G}$$
  for every Lipschitz function $v\in C(\overline{G})$ with the same boundary values as $\V$. The distance function
  $$v(x)\,=\,\min\left\{1,\frac{\mathrm{dist}(x,\partial \Omega)}{\delta}\right\}, \quad\text{where}\quad \delta\,=\, \mathrm{dist}(\Gamma,\partial \Omega)$$
  will do. Now $|\nabla v| = 1/\delta$ almost everywhere. The upper bound follows.\end{proof}

  \begin{lemma}\label{topp} We have
    $$\liminf_{x \to \xi}|\nabla \V(x)|\,\geq\,\beta\,>\,0\quad\text{whenever}\quad \xi \in \Gamma$$
    where the constant $\beta =  \mathrm{diam}(\Omega)^{-1}$.
  \end{lemma}

  \bigskip

  \begin{proof} A simple geometric reasoning provides this. Since the level curves are convex, a level set always lies entirely on one side of the tangent lines. This makes it possible to construct a linear function which lies above $\V$ in that part of $\Omega$ which is on the outer side of  a tangent and which coincides with $\V(\xi)$ at the tangent point $\xi$. The slope of the plane can be taken to be $\leq \V(\xi)/\mathrm{diam}(\Omega)$ and now $\V(\xi) = 1$. (The reader may wish to draw a picture.) Then the comparison principle yields the estimate.\end{proof}

\begin{prop}\label{beta} Let $\Gamma$ be a single point, say $\Gamma = \{0\}.$ Then
 $$
    \lim_{x \to 0}|\nabla \V(x)|\, =\, \sup_{G}\{|\nabla  \V|\}. 
$$
\end{prop}

\begin{proof} By Theorem 1 in \cite{SWY}
$$\lim_{x \to 0}\dfrac{\V(x)-1+c|x|}{|x|}\,=\,0,\qquad c\,=\,\sup_{G}\{|\nabla   \V|\},$$
and $c>0$. Let $\varepsilon > 0$. Writing $x = r(y+z)$ where $r > 0,\quad |y| < 1,$ and $|z|=1,$ we have
$$\left\vert \V \bigl(r(y+z)\bigr)-1-cr|y+z|\right\vert\,\leq \,\varepsilon r|y+z|\,<\,2 \varepsilon r$$
for $0 < r < r_{ \varepsilon}$ (= some number $< 1$). Keep $|z|=1$ fixed. Dividing out $r$ we get
\begin{equation}\label{flat}
  \sup_{B(z,1)}\left\vert\frac{\V(rx)-1}{r}-c|x|\right\vert\,<\,2 \varepsilon
  \end{equation}
when $0 < r <   r_{ \varepsilon}$. According to Theorem 2 in \cite{SWY},
%\cite{SWY}\footnote{Strictly speaking, we have to keep $|z-x|<|z|=\rho$, say, in order to prevent $x=0$. Thus the supremum is over $B(z,\rho).$ A scaling fixes this.}
inequality (\ref{flat}) implies that for any $\delta > 0$ we can find an $ \varepsilon_{\delta}$ such that 
$$\left\vert\nabla  \Bigl( \frac{\V(rx)-1}{r}\Bigr) - \nabla(c\,|x|)\right\vert_{x=z}\,<\,\delta\quad\text{when}\quad0<    \varepsilon <  \varepsilon_{\delta}  $$
which is equivalent to
$$\left\vert\nabla \V(rz)-c\frac{z}{|z|}\right\vert\,<\,\delta.$$
This holds for all $0<r<r_{\varepsilon}$ where $0<\varepsilon<  \varepsilon_{\delta}$ and hence it follows as $r \to 0$ that
$$  \lim_{x \to 0}|\nabla \V(x)|\, =\,c >0,$$
as desired. \end{proof}

\bigskip

\begin{cor}\label{singledist} Under the same assumptions as in Proposition \ref{beta},
  $$\lim_{x \to 0}|\nabla \V(x)|\,=\,\frac{1}{\mathrm{dist}(\Gamma,\partial \Omega)}.$$
\end{cor}

\bigskip

\begin{proof}  Since the function $\V$ is an optimal Lipschitz extension of its boundary data, it follows from Proposition \ref{beta} that 
$$
\|\nabla \V\|_{\infty(\Omega)}=\|\V\|_{\text{Lip}(\Gamma\cup \partial \Omega)}=\frac{1}{\dist(\Gamma,\Omega)}.
$$
\end{proof}

If $\Gamma$ is part of the High Ridge, it must be a point or a segment of a straight line. Corollary \ref{singledist} can be extended to this case. 
\begin{prop}\label{HRprop} Let $\Gamma$ be a segment on the High Ridge of $\Omega$. Then
  $$
    \lim_{x \to\xi }|\nabla \V (x)|\, =\, \frac{1}{\mathrm{dist}(\Gamma,\partial \Omega)}, \,\quad \xi \in \Gamma.
$$
\end{prop}
\begin{proof} We normalize the geometry so that $\Gamma$ is the closed segment joining the points $(\pm a,0)$ on the $x_1$-axis and $\mathrm{dist}(\xi,\partial \Omega) =1$ whenever $\xi \in \Gamma.$ Now $\V(x)\leq\mathrm{dist}(x,\partial \Omega)$. Construct the largest stadium $S$ with $\Gamma$ as its High Ridge which is contained in $\Omega.$  That is, 
$$
S=\{x\in \Omega|\, \dist(x,\Gamma)<1\}.
$$It follows by comparison
  that
  $$\mathrm{dist}(x,\partial S)\,\leq\,\V(x)\,\leq \mathrm{dist}(x,\partial \Omega),\,\quad x\in S,$$
  because $\mathrm{dist}(x,\partial S)$ is $\infty$-harmonic in $S\setminus\Gamma$.
  
  In particular, since the domain is convex, these functions coincide on a rectangle:
  \begin{equation}
  \label{rectconc}
  \mathrm{dist}(x,\partial S)\,=\,\V(x)\,= \mathrm{dist}(x,\partial \Omega)=1-|x_2|
  \end{equation}
  for $-a\,\leq\,x_1\,\leq a$ and $-1\,\leq x_2\,\leq 1$. (Draw the unit discs with centers $(\pm a,0)$ to see that the points $(\pm a,\pm 1)$ are the corners of a rectangle in $\Omega$.)
   
  As we shall see, $\V$ is glued together of three pieces (inspired by the example in Section 5 of \cite{JLM2}). Let
  $u_L$ be the solution of (\ref{maineq}) with $\Gamma=\{(-a,0)\}$. Similarly, we define $u_R$ with $\Gamma=\{(a,0)\}$. Now Corollary \ref{singledist} implies
  $$\lim_{x\to (-a,0)}|\nabla u_L(x)| \,=\,1, \quad \lim_{x\to (+a,0)}|\nabla u_R(x)| \,=\,1.$$
      We claim that in $\Omega$
      $$
      \V(x)\,=\,\begin{cases} u_R(x),\quad &x_1 \geq a\\
      \mathrm{dist}(x,\partial \Omega),\quad &a\geq x_1\geq -a\\
      u_L(x)\quad &x_1\leq -a.
      \end{cases}
      $$
      First, it is continuous. Second, it is $\infty$-harmonic in $\Omega\cap\{|x_1|>a\}$ and when $|x_1|\leq  a$ the function coincides with $V_\infty$ by \eqref{rectconc}. The desired result follows by comparison.
\end{proof}

%%%%%%%%%%%%%%%%%%%%%%%%%%%%%%%%%%%%%%%%%%%%%%%%%%%%%%%%%%%%%%

\section{Proofs of the Theorems}
\begin{proof}[~Proof of Theorem \ref{uniqueness}.] Assume that two streamlines $x_1(t)$ and $x_2(t)$   for the ascending gradient flow in equation (\ref{asc}) emerge at a point $x_{Cl} \in G$. If they intersect some level curve at the points $y_1$ and $y_2$ and $y_1\neq y_2$, then we apply the fundamental inequality (\ref{normal}) to the domain $D$ bounded by  parts of the three curves
  $x_1(t), x_2(t)$, and the level curve. Only the arcs with endpoints: $x_{Cl},y_1$, and $y_2$ count. 
  (One may think of a curved triangle). By  inequality (\ref{normal})
  $$0\,\geq \oint_{\partial D}|\nabla \V|^{p-2}\langle\nabla \V,\mathbf{n}\rangle\,ds\,=\,
  \int_{y_1}^{y_2}|\nabla \V|^{p-1}\,ds $$
  since naturally $\langle\nabla \V,\mathbf{n}\rangle = 0$ along the streamlines and
  $$\mathbf{n}\,=\,+ \frac{\nabla \V}{|\nabla \V|}$$
  is the outer unit normal along the level curve between the points $y_1$ and $y_2$. Since $\nabla \V$ is continuous, it must be identically $0$ along this level curve. This contradicts the fact that $\nabla \V \neq 0$ in $G$. Hence we must have $y_1 =y_2$ and so the streamlines coincide: $x_1(t)\equiv x_2(t).$  \end{proof}

  \bigskip

 For a curved quadrilateral bounded by the arcs of two level curves and of two streamlines we have a convenient comparison for the supremum norm of $\nabla\V$ on the level arcs. The result indicates that such quadrilaterals cannot always exist, not if the  level difference is too big.

  \begin{lemma}\label{speedlevel} Assume  that
    \begin{itemize}
    \item the points $x_1$ and $x_2$  are on the same level curve $\V = a$,
    \item  the points $y_1$ and $y_2$ both are on the higher level curve $\V= b > a$,
    \item ascending streamlines join $x_1$ with $y_1$ and $x_2$ with $y_2$.
    \end{itemize}
    Then
    \begin{equation}\label{ab}
      \|\nabla  \V\|_{\infty,\overline{y_1y_2}}\,\leq \|\nabla  \V\|_{\infty,\overline{x_1x_2}},
    \end{equation}
    that is, the lower level curve  has the larger maximum norm for the gradient.
  \end{lemma}

  \bigskip

  \begin{proof} Use inequality (\ref{normal}) on the boundary of the domain $D$ bounded by the four arcs. The streamlines do not contribute to the line integral. Along the level arcs the outer normal has the directions $\pm \nabla \V$, the minus sign being for the lower arc between $x_1$ and $x_2$. This yields
  $$ \int_{y_1}^{y_2}|\nabla \V|^{p-1}\,ds \,\leq  \int_{x_1}^{x_2}|\nabla \V|^{p-1}\,ds.$$
  Taking the $p\!-\!1\,^{\mathrm{th}}$ roots and sending $p$ to $\infty$, we arrive at inequality (\ref{ab}).\qquad\end{proof}

  \bigskip

  \begin{proof}[Proof of Theorem \ref{Cl}.] The theorem follows from the above lemma. Indeed, let $\varepsilon > 0$ be very small. There is a strip near $\Gamma$, say $\mathrm{dist}(x,\Gamma) < l_{\varepsilon}$, where $|\nabla\V| > \beta - \varepsilon$. This strip contains all sufficiently high level curves. In a neighborhood of $\xi_0$ we have $|\nabla \V| < \alpha + \varepsilon$. If two different streamlines, starting at the same level curve in this neighborhood reach the strip without joining, then it follows from inequality (\ref{ab}) that we must have
  $$\beta-\varepsilon \leq \alpha + \varepsilon,$$
  which for a small $\varepsilon$ contradicts the assumption $\beta > \alpha.$ Therefore the streamlines must have joined before reaching the top level.
  \end{proof}

  \bigskip

We now prove a localized version of Theorem \ref{Cl}, which is Corollary \ref{extra}. In order to do that, we need the following equicontinuity of streamlines.

\begin{prop}[Convergence]\label{ascoli} Suppose that a sequence of streamlines
  $$\gamma_k \,=\,\gamma_k(t), \quad 0\leq t\leq T,\qquad (k=1,2,3,...)$$
  in $G$ is given. Then the family $\{\gamma_k\}$ is equicontinuous and bounded. Furthermore, if the initial points $\gamma_k(0)$ converge to a point $a\in G$, then the streamlines converge uniformly to the streamline via $a$.
\end{prop}

\bigskip

\begin{proof} Integrating the equation
$$\frac{d \gamma_k}{dt}\,=\,\nabla \V(\gamma_k(t))$$
we see that
$$|\gamma_k(t_2)-\gamma_k(t_1)|\,=\,\left\vert \int_{t_1}^{t_2}\!\nabla \V(\gamma_k(t))\,dt\right\vert \,\leq C|t_2-t_1| $$
by Lemma \ref{nonzero}. Also
$$|\gamma_k(t)|\,=\,\left\vert \int_{0}^{t}\!\nabla \V(\gamma_k(\tau))\,d\tau\right\vert\,\leq\, Ct\,\leq\,CT.$$
Hence the family is uniformly equicontinuous and bounded.

Thus we can apply Ascoli's Theorem to find a uniformly convergent subsequence, say
$$\gamma_{k_j} \to \gamma.$$ We may take the limit under the integral sign in
$$\gamma_{k_j}(t) - \gamma_{k_j}(0)\,=\,  \int_{0}^{t}\!\nabla \V(\gamma_{k_j}(\tau))\,d\tau$$
to arrive at
$$ \gamma(t) - \gamma(0)\,=\,  \int_{0}^{t}\!\nabla \V(\gamma(\tau))\,d\tau.$$
Differentiating, we get
$$\frac{d \gamma}{dt}\,=\,\nabla \V(\gamma(t)),$$
which means that the limit curve is a streamline  and $\gamma(0) = a$.

This was for a subsequence, but using the uniqueness theorem (Theorem \ref{uniqueness})
one can deduce that also the full sequence $\gamma_k$ converges.\end{proof}

  \begin{cor}\label{extra} Suppose that a streamline $\gamma$ joins the  points $a_0$ and $b_0$ in $G$, where $a_0$ is on the lower level, i.e. $\V(a_0) \leq \V(b_0).$  If
    $$|\nabla\V(b_0)|\,  > \,|\nabla \V(a_0)|,$$
    then there is a neighborhood of $a_0$ such that every streamline starting there joins the streamline  $\gamma$   before reaching the level curve of  $b_0$.
  \end{cor}

  \bigskip

  \begin{proof} By continuity, we can find a neighborhood of $a_0$ and a neighborhood of $b_0$ such that the strict inequality above holds extended to the neighborhoods. Consider a sequence of points $a_k$ on the level curve of $a_0$ such that $a_k \to a_0$. By Proposition  \ref{ascoli} the streamlines $\gamma_k$ starting at $a_k$ converge uniformly to $\gamma$. This implies  that when  the index $k$ is big enough, the streamline starting at $a_k$ must reach the level of $b_0$ at a point inside  the upper neighborhood.  By Theorem \ref{Cl} this is possible only if the streamline has joined $\gamma$ already before reaching the upper level. (It means that all these streamlines pass via the point $b_0$.) \end{proof}

  \begin{proof}[~Proof of Theorem \ref{single}.] We may assume that $\Gamma = \{0\}$ and $\mathrm{dist}(\Gamma, \partial \Omega) = 1$ so that
  $$\lim_{x \to 0}|\nabla \V(x)|\,=\,1$$
  by  Theorem \ref{beta} and its Corollary. With this normalization $B =B(0,1)$ is the largest disk centered at $0$ which is comprised in $\Omega$. If $B \neq \Omega$, we can find a point $\xi \in \partial \Omega$ such that $\xi \not \in \overline{B}$.  Consider the streamline $x =x(t)$ from $\xi$ to the origin.
By Lemma \ref{nonzero} $|\nabla \V| \leq 1$.   We have two cases.

  If $|\nabla \V(x(t^*))| < 1$ at some point $x^* =x(t^*)$ then there is a neighborhood $U^*$ of $x^*$ where $|\nabla \V| \leq \alpha < 1$ for some suitable $\alpha$. Given a small $\varepsilon > 0$, there is a neighborhood of the top $0$ in which $|\nabla \V| > 1-\varepsilon$. If $ \varepsilon$ is so small that
  $\alpha < 1- \varepsilon$, the quadrilateral described in Lemma \ref{speedlevel} cannot exist, since inequality (\ref{ab}) is violated. This means that any two streamlines passing via the neighborhood $U^*$ must join before reaching the top.

  We are left with the case $|\nabla \V(x(t)| \equiv 1.$ Using the arclength
  $$s\,=\,\int_0^t\!|\nabla \V(x(\tau))|\,d\tau,\qquad \frac{ds}{dt}\,=\,|\nabla \V(x(t))|$$
  as parameter we see that
  \begin{align*}
    1\,=\,\V(0)-\V(\xi)\,=&\, \int_0^{T}\frac{d\V(x(t))}{dt}\,dt\,=\, \int_0^{T}\!\langle\nabla \V(x(t)),\frac{dx(t)}{dt}\rangle\,dt\\
   =&\, \int_0^{T}\!|\nabla \V(x(t))|^2\,dt \,=\, \int_0^s\!\overbrace{|\nabla \V(x(s))|}^{\equiv 1}\,ds \,=\,s
  \end{align*}
  Thus the length of the streamline from $\xi$ to $0$ is $= 1$. But that violates the requirement that $|\xi-0| > 1$. Therefore this second case is impossible.

  The proof reveals that all streamlines starting outside the inscribed disk $\overline{B}$ have Cl-points.\end{proof}

\begin{proof}[~Proof Theorem \ref{hrthm}.] The proofs follows the same lines as the proof of Theorem \ref{single}.  The only difference is that we use Proposition \ref{HRprop} instead of Proposition \ref{singledist}.
\end{proof}
  \section{The Streamlines in a Square}
  In this section, $\Omega$ is the square defined by $$-1<x_1<1,\,\,-1<x_2<1$$ and $\Gamma$ is the origin $(0,0).$ Thus $\V(0,0) = 1$. In this case the $\infty$-potential 
  $\V$ can be defined in the whole plane by reflection through the sides of the square. (The principle is the same as the Schwartz reflecion for harmonic functions.) The resulting function is $\infty$-harmonic except at the isolated points $(2m,2n),\,\, m,n=0,\pm1,\pm2,...$ The gradient
  $\nabla \V$ is now continuous except at the aforementioned points. Moreover, at the corners
  $\nabla \V(\pm1,\pm 1) = 0$
   since $\V=0$ on the sides of the square.

  Comparison  yields
  $$1-|x|\,\leq\,\V(x)\,\leq\,\mathrm{dist}(x,\partial \Omega)$$
  so that $\V$ is a linear function on the medians (= the coordinate axes).

  If $x_p = x_p(t)$ is a streamline for the $p$-harmonic function $V_p$ with the same boundary values as $\V$ so that $V_p \to \V$ as $p \to \infty$, then
  \begin{align*}
    \frac{d\,}{dt}V_p(x_p(t))\,=&\,\langle\nabla V_p,\frac{dx_p}{dt}\rangle \,=\,|\nabla V_p(x_p(t))|^2 \\
    \frac{d^2}{dt^2}V_p(x_p(t)) \,=&\, \frac{d\,}{dt}|\nabla
    V_p(x_p(t))|^2\,=\,2\,\Delta_{\infty}\!V_p(x_p(t))\\ \nonumber
    =&\,-\frac{1}{p-1}|\nabla V_p(x_p(t))|^2\Delta V_p\,\geq0, \nonumber
  \end{align*}
  since $\Delta V_p \leq 0$ (superharmonic) by Lewis's theorem. Thus the functions
  $$t \mapsto V_p(x_p(t))$$
  are convex. Unfortunately, the streamlines usually move as $p\to \infty,$ making the control of the process difficult.
  However, the diagonals are streamlines for all $p$. Thus the limit function
  $$\V(t,t) \quad \text{is convex when} \quad   -1 \leq t \leq 0$$
  on the diagonal from $(-1,-1)$ to $(0,0)$. Since the limit $V_{\infty}(t,t)$ has a continuous derivative with respect to $t$, it follows by Theorem 25.7 in \cite{R}, 
  that on the diagonal even the derivatives of $V_p$ converge uniformly.\footnote{Unfortunately, the uniform convergence $\nabla V_p \to \nabla\V$ is not known to us.} It follows that
  the speed $|\nabla \V|$ is non-decreasing along the diagonal.

  We sum up a few properties:
  \begin{enumerate}
    \item From each point on the boundary $\partial \Omega$ a unique streamline starts and terminates at the origin. Through each point there passes at least one streamline.
  \item A streamline has a continuous tangent.
  \item The diagonals and medians are streamlines.
  \item No streamline can join the medians.
  \item The speed $|\nabla\V|$ is non-decreasing on the diagonals.\footnote{It is likely that this holds on all streamlines.}
   \item There are infinitely many Cl-points near the corners.
    \item There are infinitely many Cl-points near the origin.
    \item There are infinitely many Cl-points along any streamline except the medians.
  \end{enumerate}
  This can be directly deduced from the previous results except for the three last points, which require some further explanation.

\begin{proof}[~Proof of 6)] The gradient is zero at the corners and the gradient is non-zero at all interior points.  Therefore there must be infinitely many points $a_0$ and $b_0$ near the corners satisfying the assumptions of Corollary \ref{extra}. This implies that there are infinitly many Cl-points near the corners.
\end{proof}

\begin{proof}[~Proof of 7)] We prove that in each disk around the origin, there is at least one Cl-point. The result follows from this. We assume towards a contradiction that there is $c\in (0,1)$ such that the set $\{V_\infty>c\}$ does not contain any such points. We apply Theorem \ref{single} to the restriction of $w=(\V-c)/(1-c)$ to the set $\{\V\geq c\}$ to conclude that the set $\{\V>c\}$ is a ball $B$.  In particular, $|\nabla \V|=1$ in $B$. 

Denote by $y_1$ the intersection of $B$ and the lower right diagonal. Let $x_1$ be the closest point to the  midpoint $(0,-1)$ of the lower side, such that the streamline starting at $x_1$ passes through $y_1$.\footnote{Here the notation $x=(x_1,x_2)$ is abandoned, the subindices referring to different points.} We have two alternatives: 1) $x_1$ is the corner point $(1,-1)$ and 2) $x_1$ is not the corner point (it cannot be the midpoint).

In the case of 1), any streamline starting at a point $x_2$ to the left of the corner, intersects $\partial B$ at a point $y_2\neq y_1$ which is not on the diagonal. Since we may take $x_2$ as close as we wish to the diagonal, we may assume $|\nabla V_\infty|<\frac{1}{2}$ on the line between $x_1$ and $x_2$. Moreover, on the level set joining $y_1$ and $y_2$ (= the circle $\partial B$), we have $|\nabla V_\infty|=1$. By applying Lemma \ref{speedlevel} to the pair of points $x_1,x_2$ and $y_1,y_2$, we obtain 
$$
\|\nabla V_\infty\|_{\infty,\overline{x_1\,x_2}}\,\geq\, \|\nabla V_\infty\|_{\infty,\overline{y_1\,y_2}}\, =\,  1, 
$$
which is a contradiction.

In the case of 2), let $x_2$ be a point to the left of $x_1$ and $y_2$ the corresponding point on $\partial B$. By definition, $y_2\neq y_1$. Take $z_1$ to be a point on the streamline from $x_1$ to $y_1$. Let $z_2$ be a point on the same level line  as $z_1$ and on the streamline between $x_2$ to $y_2$. By Lemma \ref{speedlevel} applied to the pair of points $y_1,y_2$ and $z_1,z_2$, we obtain that 
$$
\|\nabla V_\infty\|_{\infty,\overline{z_1\,z_2}}\geq 1.
$$
Since the pair $z_1,z_2$ is arbitrary and since we may choose $x_2$ arbitrary close to $x_1$, this implies that $|\nabla V_\infty|=1$ along the streamline starting at $x_1$. Since the distance between $x_2$ and the origin is strictly larger than 1, this is a contradiction.
\end{proof}

\begin{proof}[~Proof of 8)] Let $x$ be a boundary point which is not a midpoint of a side. Then $|x|>1$. Therefore, along any streamline starting at $x$, there must be a point $y$ where $|\nabla V_\infty|<1$. Since $|\nabla V_\infty|$ is continuous along the streamline, there must be infinitely many points $a_0$ and $b_0$ along this streamline satisfying the assumptions of Corollary \ref{extra} and therefore there are infinitly many Cl-points along this streamline.
\end{proof}
\begin{figure}[h!tp]
\begin{center}
\includegraphics[scale=0.09]{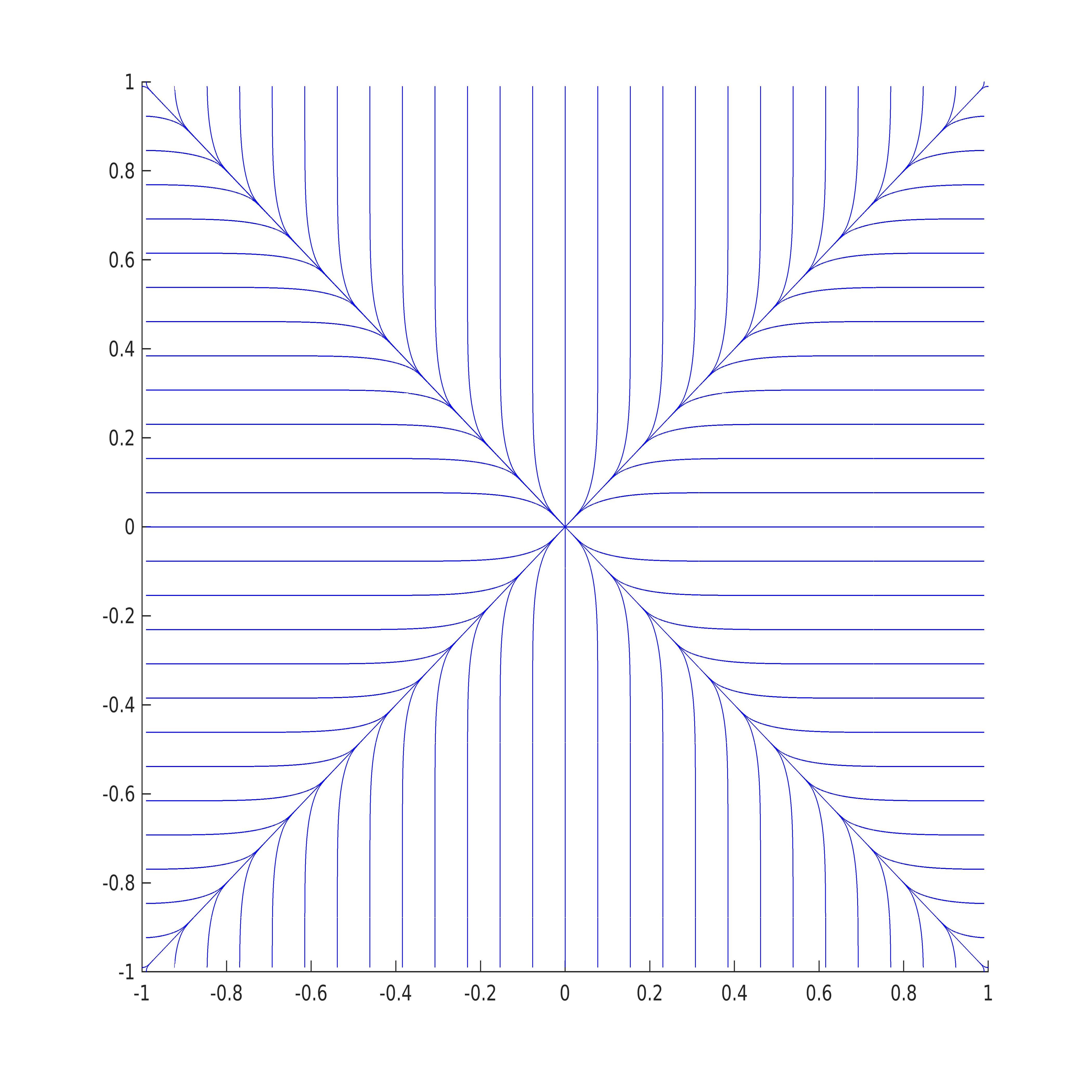}
\caption{The streamlines of $V_\infty$ when $\Omega$ is the square $-1<x_1<1,\,\,-1<x_2<1$.}
\end{center}
\end{figure}
We conjecture that every streamline except the medians joins a diagonal before reaching the midpoint and that the only Cl-points are the points on the diagonals. This is also suggested by Figure 1.

 \bigskip
  \paragraph{Epilogue.}
One may wonder whether $|\nabla \log \V| \geq 1$ in the square. This would show that $\V$ is the same function as the $\infty$-Ground State described in  Section 4 of \cite{JLM2}. This is also suggested by numerics.

  \bigskip
\paragraph{Acknowledgments:} Erik Lindgren was supported by the Swedish Research Council, grant no. 2012-3124 and 2017-03736. Peter Lindqvist was supported by The Norwegian Research Council, grant no. 250070 (WaNP). We thank the two referees for carefully reviewing this work and for pointing out a gap in the proof of Proposition \ref{fundamental}.

\bigskip
\noindent {\textsf{Erik Lindgren\\  Department of Mathematics\\ Uppsala University\\ Box 480\\
751 06 Uppsala, Sweden}  \\
\textsf{e-mail}: erik.lindgren@math.uu.se\\

\noindent \textsf{Peter Lindqvist\\ Department of
   Mathematical Sciences\\ Norwegian University of Science and
  Technology\\ N--7491, Trondheim, Norway}\\
\textsf{e-mail}: peter.lindqvist@ntnu.no


\begin{thebibliography}{amsrefs} {\small
  \bibitem[Ar1]{A1} G. Aronsson. {\it  Extension of functions satisfying Lipschitz conditions},
    Arkiv f\"{o}r Matematik \textbf{6}, 1967, pp. 551--561.
     \bibitem[Ar2]{A2} G. Aronsson. {\it On the partial differential equation} $u_x^2u_{xx}+2u_xu_yu_{xy}+u_y^2u_{yy} =0, $ Arkiv f\"{o}r Matematik \textbf{7}, 1968, pp. 397--425.
\bibitem[CF]{CF} G. Crasta, I. Fragal\`a. {\it On the characterization of some classes of proximally smooth sets},  ESAIM: Control, Optimisation and Calculus of  Variations, \textbf{22}, 3, 2016, pp. 710--727.

\bibitem[CF2]{CF2} G. Crasta, I. Fragal\`a. \emph{On the Dirichlet and Serrin problems for the inhomogeneous infinity Laplacian in convex domains: regularity and geometric results.} Archive for Rational Mechanics and Analysis 218 (2015), no. 3, 1577--1607.

\bibitem[CMS]{CMS} V. Caselles, J-M. Morel, C. Sbert. {\it An   axiomatic approach to image interpolation}, IEEE Transactions on Image Processing \textbf{7}, 1998, pp. 376--386.
\bibitem[CIL]{CIL} M. Crandall, H. Ishii, P.-L. Lions. {\it User's guide to
    viscosity solutions of second order partial differential
    equations}, Bulletin of the American Mathematical Society
  \textbf{27}, 1992, pp. 1--67.
\bibitem[EG]{EG} L. Evans, R. Gariepy. Measure Theory and Fine Propeties
  of Functions, CRC Press, Boca Raton 1992.
  \bibitem[ESa]{ESa} L. Evans, O. Savin. {\it $C^{1,\alpha}$regularity of infinite harmonic functios in two dimensions},  Calculus of Variations and Partial Differential Equations \textbf{32}, 2008, pp. 325--347.
\bibitem[ES]{ES} L. Evans, Ch. Smart. {\it Everywhere differentiability of infinity harmonic functions}, Calculus of Variations and Partial Differential Equations \textbf{42}, 2011, pp. 289--299.
%\bibitem{G}. E. Giusti.  Metodi Diretti nel Calcolo delle Variazioni,
%  Unione Matematica Italiana, Bologna 1994.
%\bibitem{G}. E. Giusti. Direct Methods in the Calculus of Variations,
  World Scientific, Singapore 2003.
  \bibitem[Ja]{Ja} U. Janfalk. {\it Behaviour in the limit, as $p \to \infty$, of minimizers of functionals involving $p$-Dirichlet integrals,} SIAM Journal on Mathematical Analysis \textbf{27}, no.2, 1996, pp. 341--360.
 \bibitem[J]{J} R. Jensen. {\it Uniqueness of Lipschitz extension: minimizing the sup norm of the gradient}, Archive for Rational Mechanics and Analysis \textbf{123}, 1993, pp. 51--74.
\bibitem[JJ]{JJ} V. Julin, P. Juutinen. {\it A new proof for the equivalence  of weak and viscosity solutions for the $p$-Laplace equation}, Communications on Partial Differential Equations \textbf{37}, no. 5, 2012, pp. 934--946.


\bibitem[JLM]{JLM} P. Juutinen, P. Lindqvist, J. Manfredi. {\it On the
    equivalence of viscosity solutions and weak solutions for a
    quasi-linear equation},  SIAM Journal on Mathematical Analysis 
    \textbf{33}, 2001, pp. 699--717.

\bibitem[JLM2]{JLM2} P. Juutinen, P. Lindqvist, J. Manfredi. {\it The infinity Laplacian: examples and observations}, Papers on analysis, Rep. Univ. Jyv\"askyl\"a Dep. Math. Stat., {\bf 83}, Univ. Jyv\"askyl\"a, Jyv\"askyl\"a, 2001, pp. 207--217.    

\bibitem[JLM3]{JLM3} P. Juutinen, P. Lindqvist, J. Manfredi.  {\it The $\infty$-eigenvalue problem}, Arch. Ration. Mech. Anal. 148 (1999), no. 2, 89--105.
    
%\bibitem{KLP}. J. Kinnunen, T. Lukkari, M. Parviainen. {\it An
%    existence result for superparabolic functions}, Journal of
  %Functional Analysis \textbf{258}, 2010, pp. 713-728.
\bibitem[KZZ]{KZZ}  H. Koch, Y. Zhang, Y. Zhou. {\it An asymptotic sharp Sobolev regularity for planar infinity harmonic functions}, Preprint 2018, arXiv:1806.01982.
\bibitem[K]{K} S. Koike. A Beginner's Guide to the Theory of Viscosity
  Solutions. (MSJ Memoirs \textbf{13}, Mathematical Society of Japan),
  Tokyo 2004.
\bibitem[L]{L}  J. Lewis. {\it Capacitory functions in convex rings}, Archive for Rational Mechanics and Analysis \textbf{66}, no. 3, 1977, pp. 201--224.
\bibitem[MPS]{MPS} J. Manfredi, A. Petrosyan, H. Shahgholian. {\it A free boundary problem for $\infty$-Laplace equation}, Calculus of Variations and Partial Differential Equations \textbf{14}, no. 3, 2002, pp. 359--384.
\bibitem[PSW]{PSW} Y. Peres, O. Schramm, S. Sheffield, D. Wilson. {\it Tug-of-war and the infinity Laplacian}, Journal of the American Mathematical Society \textbf{22}. 2009, pp. 167--210.
\bibitem[R]{R} R. Rockafeller. Convex  Analysis, Princeton University Press, USA 1970.
\bibitem[S]{S} O. Savin. {\it $C^1$ regularity for infinity harmonic functions in two dimensions}, Archive for Rational Mechanics and Analysis \textbf{176}, no. 3, 2005, pp. 351--361.
 \bibitem[SWY]{SWY} O. Savin, C. Wang, Y. Yu. {\it Asymptotic behaviour of infinity harmonic functions near an isolated singularity}, International Mathematical Research Notes \textbf{6}, 2008, rnm163.
% \bibitem{WY}. C. Wang, Y. Yu. {\it $C^1$-boundary regularity of planar infinity harmonic functions}, Mathematical Research Letters \textbf{19}, 2012, no. 4, pp. 823--835.
 %\bibitem{Y}. Y. Yu. {\it A remark on $C^2$ infinity harmonic functions}, Electronic Journal of Differential Equations 2006:\textbf{122}.   
}
\end{thebibliography}
\end{document}